\documentclass[preprint,12pt]{elsarticle}

\usepackage{amsmath,amssymb,amsthm,bm}
\usepackage{graphicx}
\usepackage{epstopdf}
\usepackage{multirow}
\usepackage{tabularx}
\usepackage{booktabs}
\usepackage{array}
\usepackage{xcolor}

\newcommand{\ii}{\mathrm{i}}
\newcommand{\dt}{\Delta t}

\newcommand{\C}{\mathbb{C}}
\newcommand{\norm}[1]{\left\lVert #1\right\rVert}

\theoremstyle{plain}
\newtheorem{theorem}{Theorem}[section]
\newtheorem{proposition}[theorem]{Proposition}
\newtheorem{corollary}[theorem]{Corollary}

\theoremstyle{remark}
\newtheorem{remark}[theorem]{Remark}

\begin{document}
\begin{frontmatter}

\title{Sharp CFL stability and temporal-dispersion optimization of symmetric splitting schemes for time-domain Maxwell equations}

\author[aff1]{Hui Duan}
\author[aff2]{Hongliang Li\corref{cor1}}
\ead{lhl@sicnu.edu.cn}
\author[aff1]{Lunzhong Guo}

\cortext[cor1]{Corresponding author.}
\affiliation[aff1]{organization={School of Big Data, Chengdu Technological University},
            city={Chengdu},
            postcode={611730},
            country={China}}
\affiliation[aff2]{organization={School of Mathematics and Sciences, Sichuan Normal University},
            city={Chengdu},
            postcode={610066},
            country={China}}

\begin{abstract}
We analyze coefficient design in a one-parameter family of explicit palindromic electric--magnetic splittings for the time-domain Maxwell equations. After fourth-order staggered spatial discretization, the Fourier amplification matrix depends on the single scalar $g_2=a(1-2a)/2$. We prove that $a=1/4$ is the unique real coefficient maximizing the spectral CFL interval, with threshold $s_*=12/(7\sqrt d)$. We then identify a real-coefficient obstruction to higher phase accuracy: cancellation of the leading temporal phase defect requires $g_2=1/12$, whereas every real member satisfies $g_2\le 1/16$. The resulting complex-conjugate coefficients give fourth-order temporal phase accuracy for each fixed semidiscrete Fourier mode and have threshold $6\sqrt3/(7\sqrt d)$, while the complete field update remains globally second order in time. For real Maxwell data, the physical output is the real projection of the complex trajectory; this projection is branch independent and preserves the second-order error bound. We further give an exactly equivalent doubled real-arithmetic realization, which clarifies the role of the auxiliary imaginary component without changing the numerical method. A semidiscrete convergence result and numerical experiments confirm the distinction between stability optimization and phase optimization.
\end{abstract}


\end{frontmatter}

\section{Introduction}\label{sec:intro}
The finite-difference time-domain (FDTD) method introduced by Yee~\cite{yee1966} remains one of the basic discretizations for transient Maxwell equations because it is explicit, local, and naturally staggered in space and time; see, for example,~\cite{taflove2005}. Two numerical effects are particularly important for long-time wave propagation. The first is the Courant--Friedrichs--Lewy (CFL) restriction on the time step. The second is numerical dispersion: even a stable nondissipative scheme can accumulate a substantial phase error over many wavelengths. These two effects are related but need not be optimized by the same time integrator.

A classical way to relax the CFL restriction is to use alternating-direction implicit or locally one-dimensional FDTD schemes~\cite{zheng1999,zheng2000,namiki2000,shibayama2005}, whose dispersion properties have also been studied extensively~\cite{zhao2002,juntunen2000,Ahmed2006}. A different line of work exploits the operator and geometric structure of Maxwell equations. Symplectic FDTD schemes~\cite{hirono1997,hirono2001,huang2008,huang2014}, operator splittings for semidiscrete Maxwell systems~\cite{Botchev2009}, and energy-conserving splittings~\cite{Chenwenbin2009,Chenwenbin2010,Cai2015} all use the fact that the evolution can be decomposed into simpler subproblems. In this setting the coefficients of a composition are genuine design variables: they determine not only formal order, but also the shape of the stability polynomial and the leading dispersive defect.

Recent Maxwell integrators have increasingly combined high order with additional structural or algorithmic properties. Energy-structure-preserving compact splittings were developed in~\cite{Kong2023}; alternating-direction splitting has been combined with isogeometric discretization and scalable solvers in~\cite{Los2023}; a sixth-order multi-structure-preserving method was proposed in~\cite{Jiang2024}; and meshless energy-conserving time splittings were studied in~\cite{Gao2025}. Very recent explicit high-order splitting schemes further combine compact spatial discretization, boundary closures, unconditional stability, energy preservation, and rigorous convergence analysis~\cite{CaiWang2026}. These developments make it important to formulate precisely the more limited question addressed by a low-stage explicit composition.

The present paper asks: \emph{for a fixed five-subflow palindromic electric--magnetic splitting, how far can stability and temporal phase accuracy be improved solely by coefficient selection?} This question is closely related to the optimized symplectic-FDTD coefficients of Kusaf, Oztoprak, and Daoud~\cite{kusaf2005}. Their work demonstrated the practical value of tuning exponential-operator coefficients. What remains useful from a numerical-analysis viewpoint is a sharp characterization of the admissible one-parameter family and, in particular, of the incompatibility between the best real-coefficient CFL interval and cancellation of the leading temporal phase error.

The main contributions are as follows.
\begin{enumerate}
\item A common $2\times2$ Fourier amplification matrix is derived for the palindromic family, reducing both stability and phase analysis to the scalar coefficient $g_2=a(1-2a)/2$.
\item For real coefficients, $a=1/4$ is proved to be the unique maximizer of the contiguous spectral CFL interval. This recovers the coefficient set used in~\cite{kusaf2005}, but adds an explicit uniqueness and sharp-threshold characterization within the family considered here.
\item The phase-matching condition is shown to require $g_2=1/12$, whereas every real member satisfies $g_2\le 1/16$. Thus no real coefficient in this family can cancel the leading temporal phase defect. Allowing complex coefficients gives an explicit conjugate pair with fourth-order temporal phase accuracy for each fixed semidiscrete Fourier mode.
\item The complete splitting map is shown to remain second order for the field variables, and a semidiscrete convergence proposition is given under a strict spectral stability margin. For the complex optimizer we justify the real output projection, prove its conjugate-branch independence, and derive an exactly equivalent $2N$-dimensional real-arithmetic formulation. Numerical tests then separate the effects of stability optimization and phase optimization.
\end{enumerate}

Complex-coefficient splitting methods are well established in geometric numerical integration~\cite{Blanes2012,Blanes2022,Bernier2023}. Here they are used only as the analytic continuation required by the phase-matching condition; they are not introduced to claim a higher global order or exact electromagnetic-energy preservation. This distinction is important because the phase-optimized map naturally evolves in the complexification of the real Maxwell phase space. Since the underlying Maxwell problem and its initial data are real, the physically relevant approximation is obtained by projecting the complex numerical state onto the real phase space at output, i.e. by taking its real part. Such an output projection is standard for complex-coefficient splittings applied to real-valued problems~\cite{Blanes2012}; below we show that it is branch independent and cannot reduce the convergence order. We also show that complex data types are not mathematically essential: the same recurrence can be written exactly as a doubled real system for the real and imaginary components. This realification is an implementation equivalence, not a new integrator and not a way of deleting the imaginary dynamics.

The remainder of the paper is organized as follows. Section~\ref{sec:framework} introduces the Maxwell system, the palindromic splitting family, and the fourth-order staggered spatial discretization. Section~\ref{sec:design} derives the common amplification matrix and gives the sharp stability, phase, and semidiscrete convergence results. Section~\ref{sec:numerics} presents numerical experiments for a homogeneous periodic problem with an exact solution. Section~\ref{sec:conclusion} summarizes the conclusions and limitations.

\section{Maxwell system and palindromic splitting family}\label{sec:framework}
\subsection{Lossless Maxwell evolution and energy scaling}
For the coefficient analysis we consider a homogeneous, linear, isotropic, and lossless medium. The source-free Maxwell equations are
\begin{equation}\label{eq:maxwell-physical}
 \mu\,\partial_t\bm H=-\nabla\times\bm E,
 \qquad
 \varepsilon\,\partial_t\bm E=\nabla\times\bm H,
\end{equation}
with the divergence constraints $\nabla\cdot(\mu\bm H)=0$ and $\nabla\cdot(\varepsilon\bm E)=0$ propagated by the evolution when they hold initially. Periodic or perfectly electrically conducting (PEC) boundary conditions are assumed whenever an energy identity is invoked. With
\[
 \bm h=\sqrt{\mu}\,\bm H,
 \qquad
 \bm e=\sqrt{\varepsilon}\,\bm E,
 \qquad
 \nu=(\varepsilon\mu)^{-1/2},
\]
we obtain
\begin{equation}\label{eq:scaled-maxwell}
 \partial_t
 \begin{bmatrix}\bm h\\ \bm e\end{bmatrix}
 =\nu
 \begin{bmatrix}
 0&-\nabla\times\\
 \nabla\times&0
 \end{bmatrix}
 \begin{bmatrix}\bm h\\ \bm e\end{bmatrix}
 =:\mathcal A
 \begin{bmatrix}\bm h\\ \bm e\end{bmatrix}.
\end{equation}
Under periodic or PEC boundary conditions, integration by parts gives the standard lossless energy law
\begin{equation}\label{eq:energy-cont}
 \frac{d}{dt}\,\frac12\left(\norm{\bm h}_{L^2}^2+\norm{\bm e}_{L^2}^2\right)=0.
\end{equation}
The coefficient analysis and the principal numerical validation below use constant material parameters. Variable coefficients, absorbing layers, and boundary closures beyond the periodic setting lie outside the scope of the present Fourier analysis.

\subsection{Electric--magnetic split and the one-parameter palindromic family}
We split $\mathcal A=\mathcal C+\mathcal D$ with
\begin{equation}\label{eq:CDsplit}
 \mathcal C=\nu\begin{bmatrix}0&0\\ \nabla\times&0\end{bmatrix},
 \qquad
 \mathcal D=\nu\begin{bmatrix}0&-\nabla\times\\ 0&0\end{bmatrix}.
\end{equation}
The block structure implies $\mathcal C^2=\mathcal D^2=0$, and therefore each subflow is explicit:
\begin{equation}\label{eq:subflows}
 e^{\tau\mathcal C}=I+\tau\mathcal C,
 \qquad
 e^{\tau\mathcal D}=I+\tau\mathcal D.
\end{equation}
We study the five-subflow palindromic composition
\begin{equation}\label{eq:Phi-a}
 \Phi_{\dt}(a)
 =e^{a\dt\mathcal C}
  e^{\frac12\dt\mathcal D}
  e^{(1-2a)\dt\mathcal C}
  e^{\frac12\dt\mathcal D}
  e^{a\dt\mathcal C},
 \qquad a\in\C.
\end{equation}
Thus there are three $\mathcal C$-updates and two $\mathcal D$-updates. The palindromic form gives time symmetry, while the coefficient sums are one. Consequently $\Phi_{\dt}(a)$ is a consistent symmetric method and hence is at least second order for every fixed $a$ for which the subflows are defined~\cite{mclachlan2002,hairer2006}. For $a=1/2$, the central $\mathcal C$-flow is the identity and the two adjacent half $\mathcal D$-flows combine, so~\eqref{eq:Phi-a} reduces to Strang splitting.

\begin{remark}
The terminology ``three-stage'' is used in parts of the symplectic-FDTD literature for closely related propagators, but it can obscure the actual work count. We therefore use ``five-subflow palindromic composition'' throughout and state the three-$\mathcal C$/two-$\mathcal D$ structure explicitly.
\end{remark}

\subsection{Fourth-order staggered spatial discretization}
On a Yee-type staggered grid, we use the fourth-order centered approximation
\begin{equation}\label{eq:spatialdis-new}
 \partial_x u(x_i,\cdot)
 =\frac{27\bigl(u_{i+1/2}-u_{i-1/2}\bigr)-\bigl(u_{i+3/2}-u_{i-3/2}\bigr)}{24\,\Delta x}
 +\mathcal O(\Delta x^4),
\end{equation}
with analogous formulas in the other coordinate directions. The corresponding dimensionless modified wave number is
\begin{equation}\label{eq:eta}
 \eta_\xi(\kappa_\xi)
 =\frac{27\sin(\kappa_\xi\Delta x/2)-\sin(3\kappa_\xi\Delta x/2)}{12},
 \qquad -\frac73\le \eta_\xi\le\frac73,
\end{equation}
for an equal mesh width $\Delta x$ in each active direction. In $d$ spatial dimensions define
\begin{equation}\label{eq:Kh-q}
 K_h^2=\frac{1}{\Delta x^2}\sum_{\xi=1}^{d}\eta_\xi^2,
 \qquad
 q=\nu\dt K_h,
 \qquad
 s=\frac{\nu\dt}{\Delta x}.
\end{equation}
Then
\begin{equation}\label{eq:qbound}
 0\le q^2\le \frac{49d}{9}s^2.
\end{equation}
The coefficient optimization below depends on the spatial discretization only through $K_h$ and the bound~\eqref{eq:qbound}; a different spatial stencil can be inserted by replacing its modified wave number.

\section{Sharp coefficient design}\label{sec:design}
\subsection{Common amplification matrix}
For a fixed nonzero transverse Fourier mode, the energy-scaled semidiscrete Maxwell system reduces, for each polarization, to the harmonic oscillator
\begin{equation}\label{eq:oscillator}
 \frac{d}{dt}\begin{bmatrix}h\\ e\end{bmatrix}
 =\nu K_h
 \begin{bmatrix}0&-1\\ 1&0\end{bmatrix}
 \begin{bmatrix}h\\ e\end{bmatrix}.
\end{equation}
Applying~\eqref{eq:Phi-a} gives the one-step matrix
\begin{equation}\label{eq:Ma}
 M_a(q)=
 \begin{bmatrix}
 A_a(q)&B_a(q)\\
 C_a(q)&A_a(q)
 \end{bmatrix},
 \qquad \det M_a(q)=1,
\end{equation}
where
\begin{align}
 A_a(q)&=1-\frac{q^2}{2}+\frac{g_2(a)}{2}q^4,
 &g_2(a)&=\frac{a(1-2a)}{2},\label{eq:A-g2}\\
 B_a(q)&=-q+\frac{1-2a}{4}q^3,\label{eq:B}\\
 C_a(q)&=q+a(a-1)q^3+\frac{a^2(1-2a)}{4}q^5.\label{eq:C}
\end{align}
Hence the characteristic polynomial is
\begin{equation}\label{eq:charpoly}
 \lambda^2-2A_a(q)\lambda+1=0.
\end{equation}
Whenever $A_a(q)$ is real, the two eigenvalues lie on the unit circle if $|A_a(q)|\le1$. The CFL and dispersion calculations therefore reduce to the scalar function $A_a$.

\begin{remark}
The condition $|A_a|\le1$ is a spectral von Neumann condition. At an endpoint where $A_a=\pm1$, the matrix may be defective; power-boundedness then requires separate inspection. For the upper CFL endpoints obtained below, the extremal nonzero Fourier mode is defective. We therefore report the endpoint as a CFL \emph{threshold} and use a strict inequality for robust time stepping.
\end{remark}

\subsection{Sharp optimization of the real-coefficient CFL interval}
For real $a$,
\begin{equation}\label{eq:g2-real-bound}
 g_2(a)=\frac{a(1-2a)}{2}
 =\frac1{16}-\left(a-\frac14\right)^2
 \le\frac1{16},
\end{equation}
with equality only for $a=1/4$.

\begin{theorem}[unique real CFL maximizer]\label{thm:stab}
Within the real family~\eqref{eq:Phi-a}, the unique coefficient that maximizes the contiguous spectral stability interval containing $q=0$ is
\begin{equation}\label{eq:astab}
 a=\frac14,
 \qquad 1-2a=\frac12.
\end{equation}
For this choice, $|A_a(q)|\le1$ for $0\le q^2\le16$. With the fourth-order staggered stencil~\eqref{eq:spatialdis-new}, the corresponding $d$-dimensional CFL threshold is
\begin{equation}\label{eq:s-stab}
 s_{\rm stab}^*=\frac{12}{7\sqrt d}.
\end{equation}
The practical stability condition is $s<s_{\rm stab}^*$.
\end{theorem}

\begin{proof}
Set $y=q^2$. If $a=1/4$, then $g_2=1/16$ and
\[
 A_{1/4}(y)=1-\frac y2+\frac{y^2}{32}
 =-1+\frac{(y-8)^2}{32}.
\]
Thus $-1\le A_{1/4}(y)\le1$ for $0\le y\le16$, and $A_{1/4}(y)>1$ for $y>16$. If $a\ne1/4$, then~\eqref{eq:g2-real-bound} gives $g_2<1/16$, and at $y=8$,
\[
 A_a(8)=1-4+32g_2<-1.
\]
Hence no other real coefficient can have a stable interval containing $[0,16]$; indeed it loses spectral stability before $y=8$. Therefore $a=1/4$ is the unique maximizer. Finally,~\eqref{eq:qbound} with $q^2<16$ gives~\eqref{eq:s-stab}.
\end{proof}

The coefficient set~\eqref{eq:astab} coincides with the optimized symplectic-FDTD coefficients reported in~\cite{kusaf2005}. The new point needed here is the sharp one-parameter extremal statement, which will also be contrasted with the phase-matching condition below.

\subsection{Temporal phase matching and a real-coefficient obstruction}
The exact semidiscrete oscillator~\eqref{eq:oscillator} advances by a rotation with phase $q$, so its trace satisfies $\frac12\operatorname{tr}e^{qJ}=\cos q$. For the splitting method define the numerical phase $\theta_a(q)$ by
\begin{equation}\label{eq:phase-def}
 \cos\theta_a(q)=A_a(q),
 \qquad \theta_a(q)\sim q\quad(q\to0).
\end{equation}
Expanding~\eqref{eq:phase-def} gives
\begin{equation}\label{eq:phase-expansion}
 \theta_a(q)
 =q+\left(\frac1{24}-\frac{g_2(a)}{2}\right)q^3+\mathcal O(q^5).
\end{equation}
Thus the leading temporal phase defect vanishes precisely when
\begin{equation}\label{eq:g2-disp}
 g_2=\frac1{12}.
\end{equation}

\begin{theorem}[real-coefficient barrier and complex phase optimizer]\label{thm:disp}
No real coefficient $a$ in~\eqref{eq:Phi-a} can cancel the leading temporal phase defect. Over $a\in\C$, condition~\eqref{eq:g2-disp} has the two conjugate solutions
\begin{equation}\label{eq:adisp}
 a_{\pm}=\frac14\pm\frac{\ii}{4\sqrt3},
 \qquad
 1-2a_{\pm}=\frac12\mp\frac{\ii}{2\sqrt3}.
\end{equation}
For either choice,
\begin{equation}\label{eq:phase-fourth}
 \theta_{a_\pm}(q)=q-\frac{q^5}{720}+\mathcal O(q^7).
\end{equation}
Consequently, for a fixed semidiscrete wave number $K_h$, the temporal frequency error satisfies
\begin{equation}\label{eq:frequency-fourth}
 \frac{\theta_{a_\pm}(q)}{\dt}-\nu K_h
 =-\frac{(\nu K_h)^5}{720}\dt^4+\mathcal O(\dt^6).
\end{equation}
\end{theorem}

\begin{proof}
The real-coefficient obstruction follows immediately from~\eqref{eq:g2-real-bound}, since $1/16<1/12$. Solving $a(1-2a)/2=1/12$ gives~\eqref{eq:adisp}. For $g_2=1/12$,
\[
 A_a(q)=1-\frac{q^2}{2}+\frac{q^4}{24}
 =\cos q+\frac{q^6}{720}+\mathcal O(q^8).
\]
Substitution of $\theta=q+cq^5+\mathcal O(q^7)$ into $\cos\theta=A_a(q)$ yields $c=-1/720$, proving~\eqref{eq:phase-fourth} and~\eqref{eq:frequency-fourth}.
\end{proof}

\begin{proposition}[spectral CFL threshold of the phase-optimized map]\label{prop:dispstab}
For $a=a_\pm$, $A_a(q)$ is real and
\[
 A_a(q)=1-\frac{q^2}{2}+\frac{q^4}{24}.
\]
Its nonzero modes satisfy $|A_a(q)|<1$ for $0<q^2<12$. Hence the fourth-order staggered discretization has the spectral CFL threshold
\begin{equation}\label{eq:s-disp}
 s_{\rm disp}^*=\frac{6\sqrt3}{7\sqrt d},
\end{equation}
with robust time stepping requiring $s<s_{\rm disp}^*$.
\end{proposition}

\begin{proof}
With $y=q^2$, $A(y)=1-y/2+y^2/24$ has its minimum $A(6)=-1/2$ and returns to $A=1$ at $y=12$. Thus $|A|<1$ for $0<y<12$. Combining $q^2<12$ with~\eqref{eq:qbound} gives~\eqref{eq:s-disp}. At $y=12$, $A=1$ while the off-diagonal entries in~\eqref{eq:Ma} do not both vanish, so the endpoint matrix is defective.
\end{proof}

\subsection{Global order, power boundedness, and interpretation}\label{sec:scope}
The improvement in~\eqref{eq:frequency-fourth} concerns the temporal phase of a fixed semidiscrete Fourier mode. It does not raise the order of the complete composition. Indeed, from~\eqref{eq:B},
\begin{equation}\label{eq:B-vs-exact}
 B_a(q)=-q+\frac{1-2a}{4}q^3,
 \qquad
 -\sin q=-q+\frac16q^3+\mathcal O(q^5),
\end{equation}
and the $q^3$ coefficients do not agree for $a=a_\pm$. Hence the one-step state error remains $\mathcal O(q^3)$, as expected for a symmetric second-order splitting.

The following proposition makes the corresponding semidiscrete convergence statement explicit. It is intentionally formulated for a fixed spatial grid; no uniform-in-$h$ PDE error estimate is claimed.

\begin{proposition}[second-order convergence on a fixed spatial grid]\label{prop:semidiscrete-convergence}
Assume periodic boundary conditions, constant $\varepsilon$ and $\mu$, and a fixed staggered spatial grid. Let $U_h(t)$ denote the solution of the resulting semidiscrete Maxwell system and let $U_h^n$ be generated by either $\Phi_{\rm stab}$ or $\Phi_{\rm disp}$. Suppose that all nonzero Fourier modes satisfy
\[
 0\le q\le q_0<q_*,
 \qquad
 q_*=4\quad\text{for }\Phi_{\rm stab},
 \qquad
 q_*=\sqrt{12}\quad\text{for }\Phi_{\rm disp}.
\]
Then, for every fixed final time $T>0$, there exists a constant $C_{h,T,q_0}$ independent of $\dt$ such that
\begin{equation}\label{eq:semidiscrete-global-error}
 \max_{0\le n\dt\le T}
 \norm{U_h^n-U_h(t_n)}_h
 \le C_{h,T,q_0}\,\dt^2\norm{U_h(0)}_h .
\end{equation}
For $\Phi_{\rm disp}$, $\norm{\cdot}_h$ denotes the complex discrete $\ell^2$ norm.
\end{proposition}

\begin{proof}
On a fixed grid the spatially discrete operators are finite-dimensional. Consistency of the coefficient sums and palindromic symmetry imply, by the standard splitting expansion, a local defect of order $\mathcal O(\dt^3)$ in operator norm. It remains to control repeated applications of the one-step map.

For each nonzero transverse Fourier mode with $A_a(q)\neq\pm1$, write $A_a(q)=\cos\theta_a(q)$ with $0<\theta_a(q)<\pi$. The Cayley--Hamilton identity gives
\begin{equation}\label{eq:power-formula}
 M_a(q)^n
 =\cos(n\theta_a)I
 +\frac{\sin(n\theta_a)}{\sin\theta_a}
   \bigl(M_a(q)-A_a(q)I\bigr).
\end{equation}
Near $q=0$, the ratios $B_a(q)/\sin\theta_a(q)$ and $C_a(q)/\sin\theta_a(q)$ are bounded because $B_a(q)=-q+\mathcal O(q^3)$, $C_a(q)=q+\mathcal O(q^3)$, and $\sin\theta_a(q)=q+\mathcal O(q^3)$. For $\Phi_{\rm stab}$ there is one additional interior degeneracy at $q^2=8$; direct substitution gives $M_{1/4}(\sqrt8)=-I$, while both off-diagonal entries vanish linearly with $q^2-8$, so the bound extends across this point. For $\Phi_{\rm disp}$ no such interior degeneracy occurs. Since $q_0$ stays strictly below the upper threshold, Formula~\eqref{eq:power-formula} therefore gives a uniform power bound for every Fourier block. The zero modes are exact. A unitary discrete Fourier transform then gives a power bound for the full grid operator. Finally, the usual telescoping argument sums $O(T/\dt)$ local defects of size $\mathcal O(\dt^3)$ and proves~\eqref{eq:semidiscrete-global-error}.
\end{proof}

\begin{remark}
Proposition~\ref{prop:semidiscrete-convergence} is a time-discretization result for the semidiscrete problem. A uniform estimate of the form $\mathcal O(\dt^2+h^4)$ for the continuous Maxwell solution would require a separate regularity and spatial-consistency argument and is not asserted here.
\end{remark}

For the complex phase-optimized method, the computed trajectory belongs to the complexification of the real Maxwell phase space. The following consequence identifies the real-valued output used in the numerical experiments.
\begin{corollary}[real projection of the complex phase-optimized solution]\label{cor:real-projection}
Let $U_{h,+}^n$ and $U_{h,-}^n$ be generated from the same real initial data by the two conjugate choices $a_+$ and $a_-$ in~\eqref{eq:adisp}. Then
\begin{equation}\label{eq:conjugate-branches}
 U_{h,-}^n=\overline{U_{h,+}^n},
 \qquad
 \widehat U_h^n:=\operatorname{Re}U_{h,+}^n
 =\frac12\left(U_{h,+}^n+U_{h,-}^n\right).
\end{equation}
Thus $\widehat U_h^n$ is real and independent of the selected conjugate branch. Moreover, because the exact semidiscrete solution $U_h(t_n)$ is real,
\begin{equation}\label{eq:projection-error-bound}
 \norm{\widehat U_h^n-U_h(t_n)}_h
 \leq \norm{U_{h,+}^n-U_h(t_n)}_h.
\end{equation}
Consequently the projected physical approximation $\widehat U_h^n$ retains at least the second-order convergence of Proposition~\ref{prop:semidiscrete-convergence}.
\end{corollary}
\begin{proof}
The spatially discrete Maxwell operators are real. Hence complex conjugation changes the $a_+$ composition into the $a_-$ composition, and real initial data give the first identity in~\eqref{eq:conjugate-branches} by induction over the time steps. The second identity is the elementary relation $\operatorname{Re}z=(z+\overline z)/2$. Finally, $U_h(t_n)$ is real and $\operatorname{Re}$ is a norm-one projection on the complexified discrete $\ell^2$ space, which gives~\eqref{eq:projection-error-bound}.
\end{proof}

Corollary~\ref{cor:real-projection} concerns an \emph{output projection}: the complete complex recurrence is carried out first, and only the reported physical field is replaced by its real part. Projecting after every substep or every time step would define a different numerical method and is not analyzed here. The discarded imaginary component is therefore best viewed as an auxiliary numerical component of the complexified trajectory; it may be monitored as a diagnostic but is not part of the physical Maxwell field. The determinant condition $\det M_a=1$ does not by itself imply exact preservation of the physical Maxwell energy. The present coefficient optimization consequently has a different objective from energy-preserving splittings such as~\cite{Cai2015,Kong2023,Gao2025,CaiWang2026}.

\begin{proposition}[exact real-arithmetic realization]\label{prop:realification}
Let $Z$ be a complex matrix or a complex linear operator on the complexification of a real discrete space, and define its realification by
\begin{equation}\label{eq:realification-map}
 \mathfrak R(Z):=
 \begin{bmatrix}
  \operatorname{Re}Z&-\operatorname{Im}Z\\
  \operatorname{Im}Z& \operatorname{Re}Z
 \end{bmatrix}.
\end{equation}
For $U=X+\ii Y$ identify $U$ with the doubled real vector
$\mathcal U=(X^{\mathsf T},Y^{\mathsf T})^{\mathsf T}$.  Then
\begin{equation}\label{eq:realification-action}
 \mathfrak R(Z)\begin{bmatrix}X\\Y\end{bmatrix}
 =\begin{bmatrix}\operatorname{Re}(ZU)\\\operatorname{Im}(ZU)\end{bmatrix},
 \qquad
 \mathfrak R(Z_1Z_2)=\mathfrak R(Z_1)\mathfrak R(Z_2).
\end{equation}
Consequently the complex phase-optimized splitting $\Phi_{\rm disp}$ has an exactly equivalent $2N$-dimensional real representation
\begin{equation}\label{eq:realified-Phi}
 \widehat\Phi_{\rm disp}=\mathfrak R(\Phi_{\rm disp}),
 \qquad
 \widehat\Phi_{\rm disp}^{\,n}=\mathfrak R(\Phi_{\rm disp}^{\,n}).
\end{equation}
If the initial Maxwell vector $U_h^0$ is real, then
\begin{equation}\label{eq:realified-output}
 \widehat\Phi_{\rm disp}^{\,n}
 \begin{bmatrix}U_h^0\\0\end{bmatrix}
 =\begin{bmatrix}
  \operatorname{Re}(\Phi_{\rm disp}^{\,n}U_h^0)\\
  \operatorname{Im}(\Phi_{\rm disp}^{\,n}U_h^0)
 \end{bmatrix}.
\end{equation}
Thus the first $N$ real components are exactly the projected physical field of Corollary~\ref{cor:real-projection}, with no approximation introduced by the real-arithmetic reformulation.
\end{proposition}

\begin{proof}
Formula~\eqref{eq:realification-action} follows by writing $Z=Z_R+\ii Z_I$ and $U=X+\ii Y$ and collecting real and imaginary parts.  The multiplicative identity in~\eqref{eq:realification-action} then follows directly, and applying it to the ordered product of the five subflows gives~\eqref{eq:realified-Phi}.  Repeated application gives $\mathfrak R(\Phi_{\rm disp}^n)=\mathfrak R(\Phi_{\rm disp})^n$, and~\eqref{eq:realified-output} follows because the initial imaginary component is zero.
\end{proof}

The realified substeps are particularly simple because $\mathcal C^2=\mathcal D^2=0$.  For a complex $\mathcal C$-coefficient $c_m=p_m+\ii q_m$,
\begin{equation}\label{eq:realified-C}
 \mathfrak R\!\left(I+c_m\dt\mathcal C\right)
 =
 \begin{bmatrix}
 I+p_m\dt\mathcal C&-q_m\dt\mathcal C\\
 q_m\dt\mathcal C&I+p_m\dt\mathcal C
 \end{bmatrix}.
\end{equation}
Writing the scaled fields as $\bm h=\bm h_R+\ii\bm h_I$ and
$\bm e=\bm e_R+\ii\bm e_I$, the corresponding electric update is
\begin{align}\label{eq:realified-C-fields}
 \bm e_R^+&=\bm e_R+\nu\dt\left(p_m\nabla\times\bm h_R-q_m\nabla\times\bm h_I\right),\\
 \bm e_I^+&=\bm e_I+\nu\dt\left(q_m\nabla\times\bm h_R+p_m\nabla\times\bm h_I\right),
\end{align}
while $\bm h_R$ and $\bm h_I$ are unchanged.  In the present family the two nontrivial $\mathcal D$-coefficients are real, $d_1=d_2=1/2$, so the magnetic updates decouple:
\begin{align}\label{eq:realified-D-fields}
 \bm h_R^+&=\bm h_R-\nu d_m\dt\,\nabla\times\bm e_R,\\
 \bm h_I^+&=\bm h_I-\nu d_m\dt\,\nabla\times\bm e_I,
\end{align}
with $\bm e_R$ and $\bm e_I$ unchanged.  For the $a_+$ branch the three $\mathcal C$ coefficients have
\begin{equation}\label{eq:realified-coeffs}
 (p_1,q_1)=\left(\frac14,\frac{1}{4\sqrt3}\right),\quad
 (p_2,q_2)=\left(\frac12,-\frac{1}{2\sqrt3}\right),\quad
 (p_3,q_3)=(p_1,q_1),
\end{equation}
while the $a_-$ branch reverses the signs of all $q_m$.

\begin{remark}[why the imaginary coefficients cannot simply be deleted]\label{rem:not-delete-imag}
The realification above retains the variables $(\bm h_I,\bm e_I)$ and their coupling to the real variables.  Setting $q_m=0$ in~\eqref{eq:realified-C-fields} is therefore \emph{not} an implementation of the same method.  For the phase-optimized coefficients it replaces
$a_\pm=1/4\pm\ii/(4\sqrt3)$ by $a=1/4$, hence changes $g_2$ from $1/12$ to $1/16$ and reduces $\Phi_{\rm disp}$ exactly to the real stability-optimized member $\Phi_{\rm stab}$.  Equivalently, in one-step matrix notation, if $M=A+\ii B$, then
\[
 \operatorname{Re}(M^2)=A^2-B^2\neq A^2=(\operatorname{Re}M)^2
\]
in general.  The auxiliary imaginary component generated at one step therefore feeds back into the real component at later steps and is part of the phase-cancellation mechanism.
\end{remark}

\begin{remark}[arithmetic cost]\label{rem:realified-cost}
The realified formulation removes the need for a complex data type but not the intrinsic cost of the complex method.  It stores two real copies of every field component and applies the real spatial stencil to both copies.  Its storage and leading arithmetic work are therefore comparable to a native complex implementation and are roughly doubled relative to the corresponding real-coefficient splitting.  The main benefit is conceptual and implementation portability: the auxiliary role of the imaginary component and its coupling to the physical real component are completely explicit.
\end{remark}

For reference, Table~\ref{tab:coefficients} summarizes the three members used in the numerical section. The Yee CFL value is listed separately because its spatial stencil differs from~\eqref{eq:spatialdis-new}.
\begin{table}[htbp]
\centering
\caption{Coefficient choices in the palindromic family~\eqref{eq:Phi-a}. The two-dimensional thresholds use the fourth-order staggered spatial stencil.}\label{tab:coefficients}
\begin{tabular}{lllll}
\toprule
Method & $a$ & $g_2$ & threshold in $q^2$ & $s^*$ for $d=2$\\
\midrule
Strang & $1/2$ & $0$ & $4$ & $6/(7\sqrt2)\approx0.6061$\\
$\Phi_{\rm stab}$ & $1/4$ & $1/16$ & $16$ & $12/(7\sqrt2)\approx1.2122$\\
$\Phi_{\rm disp}$ & $1/4\pm\ii/(4\sqrt3)$ & $1/12$ & $12$ & $6\sqrt3/(7\sqrt2)\approx1.0498$\\
\bottomrule
\end{tabular}
\end{table}
The standard two-dimensional Yee scheme with its second-order staggered spatial difference has the familiar threshold $s^*=1/\sqrt2\approx0.7071$~\cite{sullivan2000}. For the common fourth-order spatial stencil used by the splitting methods, $\Phi_{\rm stab}$ doubles the Strang CFL threshold, whereas $\Phi_{\rm disp}$ enlarges it by the factor $\sqrt3$. The leading phase coefficient in~\eqref{eq:phase-expansion} is, respectively, $1/24$, $1/96$, and $0$ for Strang, $\Phi_{\rm stab}$, and $\Phi_{\rm disp}$.

\section{Numerical experiments}\label{sec:numerics}
The numerical experiments are designed to illustrate the three properties analyzed in Section~\ref{sec:design}: the distinct CFL thresholds, second-order field convergence, and the reduction of temporal phase error. We deliberately use the homogeneous periodic problem for which the Fourier analysis applies directly, so that the numerical evidence tests the stated theory without introducing separate interface or boundary-closure effects. All splitting computations in this section use the fourth-order staggered spatial derivative~\eqref{eq:spatialdis-new}. No claim of exact discrete energy conservation is made.

For the real-coefficient methods we write $\widehat U_h^n=U_h^n$. For $\Phi_{\rm disp}$ we may use either native complex arithmetic or the exactly equivalent doubled real formulation of Proposition~\ref{prop:realification}. In both representations the complete auxiliary state is retained through all substeps and all time steps; no projection is fed back into the recurrence. The physical output is
\begin{equation}\label{eq:physical-output}
 \widehat U_h^n=\operatorname{Re}U_h^n,
\end{equation}
as justified by Corollary~\ref{cor:real-projection}. In the doubled real implementation this is simply the first block $X_h^n$. We compare the real exact field with $\widehat U_h^n$ in the energy-scaled discrete norm
\begin{equation}\label{eq:error-norm}
\begin{aligned}
 \norm{U-\widehat U_h}_{h,\varepsilon,\mu}^2
 =\sum_{j,k}\bigl(&\mu\,|H_x-\widehat H_{x,h}|^2
 +\mu\,|H_z-\widehat H_{z,h}|^2\\
 &+\varepsilon\,|E_y-\widehat E_{y,h}|^2\bigr)\Delta x\Delta z.
\end{aligned}
\end{equation}
For $\Phi_{\rm disp}$ we additionally monitor $\|\operatorname{Im}U_h^n\|_h$ and, when useful, the auxiliary complex-state error $\|U_h^n-U_h(t_n)\|_h$. These quantities diagnose the complexified trajectory but are not used as the physical field error.

\subsection{Periodic wave: stability, convergence, and dispersion}\label{sec:num-periodic}
Consider the two-dimensional TM system
\begin{align}\label{eq:TM-periodic}
 \partial_t H_x&=\frac1\mu\partial_z E_y, &
 \partial_t H_z&=-\frac1\mu\partial_x E_y,\\
 \partial_t E_y&=\frac1\varepsilon\left(\partial_z H_x-\partial_x H_z\right).
\end{align}
On $\Omega=(0,1)^2$ with periodic boundary conditions and $\varepsilon=\mu=1$, an exact solution is
\begin{align}
 H_x&=\frac1{\sqrt2}\sin(2\pi x)\cos(2\pi z)\sin(2\sqrt2\pi t),\\
 H_z&=-\frac1{\sqrt2}\cos(2\pi x)\sin(2\pi z)\sin(2\sqrt2\pi t),\\
 E_y&=\sin(2\pi x)\sin(2\pi z)\cos(2\sqrt2\pi t).
\end{align}

We first take $N=100$, so that $\Delta x=\Delta z=0.01$, and choose $\Delta t=0.8\Delta x$, i.e. $s=0.8$. This value exceeds both the two-dimensional Yee threshold $1/\sqrt2$ and the fourth-order-Strang threshold $6/(7\sqrt2)$, while remaining below the thresholds of $\Phi_{\rm stab}$ and $\Phi_{\rm disp}$ in Table~\ref{tab:coefficients}. Figures~\ref{fig:1numesplit} and~\ref{fig:1numyee} therefore provide a direct test of the predicted separation of the stability intervals.

\begin{figure}[htbp]
\centering
\includegraphics[width=0.80\textwidth]{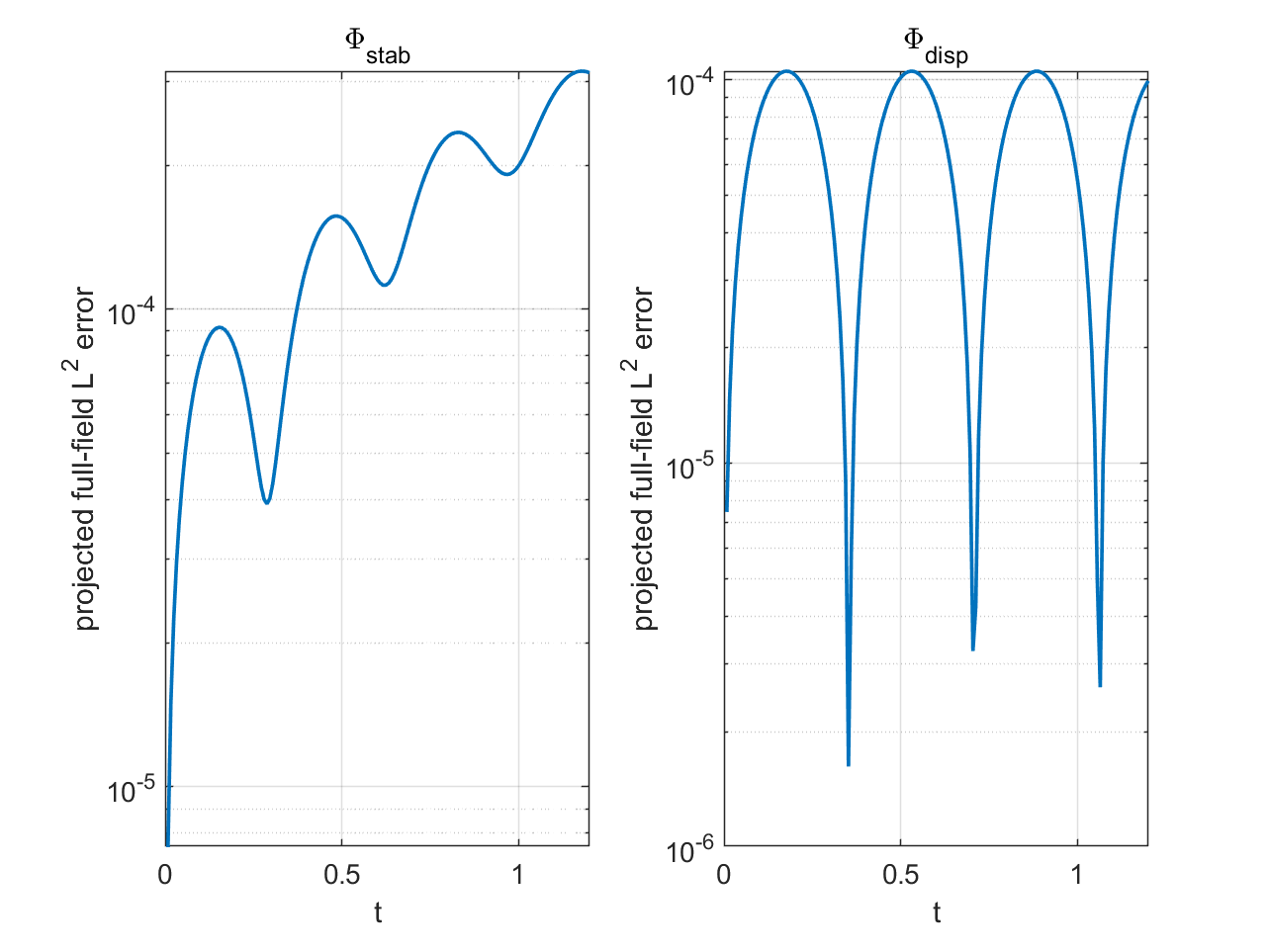}
\caption{Field-error histories for $\Phi_{\rm stab}$ (left) and $\Phi_{\rm disp}$ (right), with $N=100$, fourth-order staggered spatial discretization, and CFL number $s=0.8$. For $\Phi_{\rm disp}$ the reported field is the output projection $\operatorname{Re}U_h^n$ from~\eqref{eq:physical-output}.}
\label{fig:1numesplit}
\end{figure}

\begin{figure}[htbp]
\centering
\includegraphics[width=0.80\textwidth]{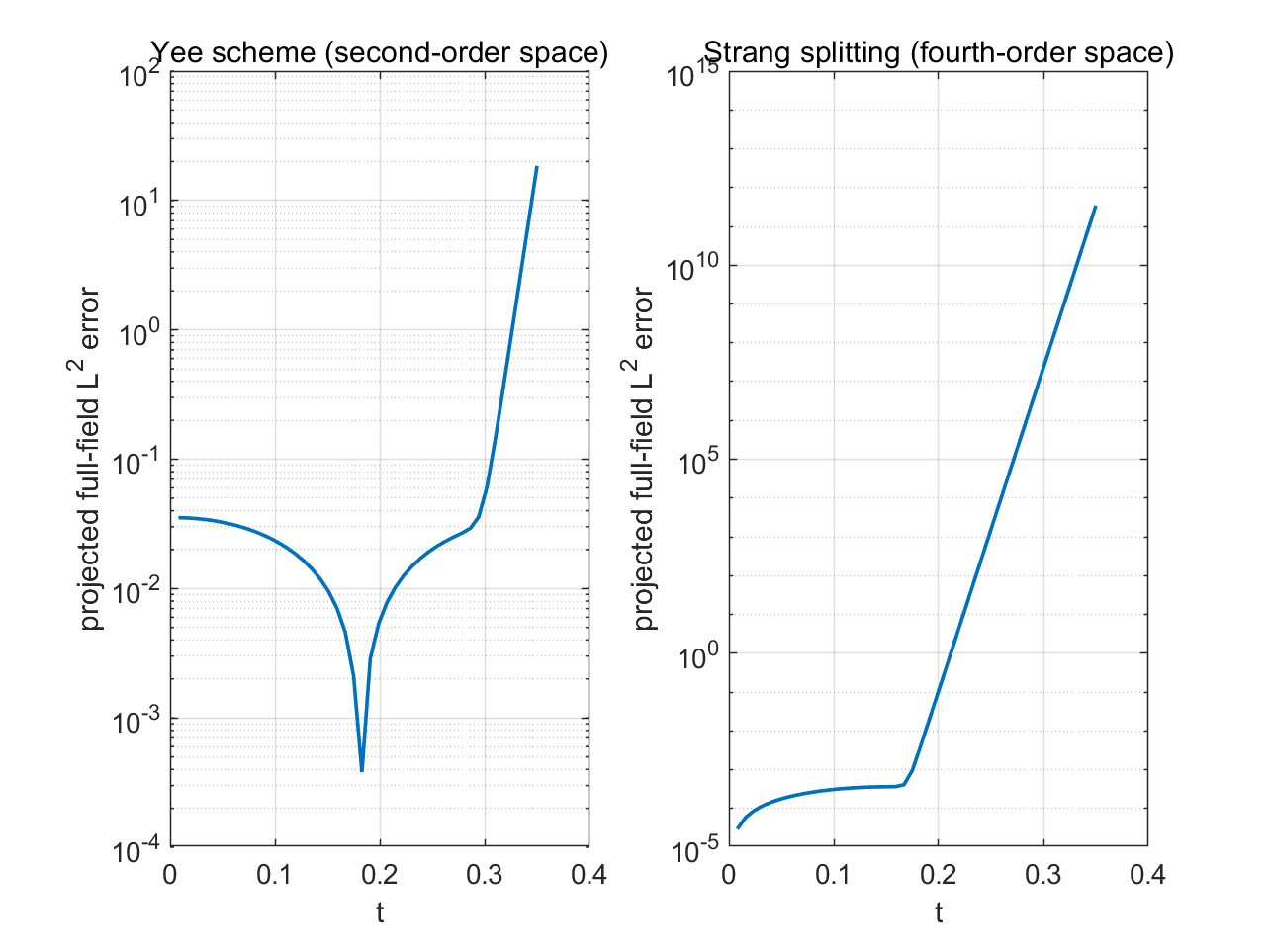}
\caption{Field-error histories for the Yee scheme (left) and Strang splitting with the fourth-order spatial stencil (right), both computed with $s=0.8$. Since this CFL number exceeds their respective two-dimensional stability thresholds, both methods become unstable.}
\label{fig:1numyee}
\end{figure}

To verify the temporal convergence order, we fix the spatial grid at $N=80$ and vary the time step by taking $s=0.4,0.2,0.1,0.05$. The temporal convergence plots in Fig.~\ref{fig:1conver} should be interpreted together with Proposition~\ref{prop:semidiscrete-convergence} and Corollary~\ref{cor:real-projection}. Both optimized compositions are second order for the complete physical field. For $\Phi_{\rm disp}$ the comparison is made after the output projection $\operatorname{Re}U_h^n$. The smaller error observed over the plotted range is consistent with cancellation of the leading temporal phase defect and possible cancellation of imaginary error components under projection, not with a fourth-order field-convergence theorem.

\begin{figure}[htbp]
\centering
\includegraphics[width=0.80\textwidth]{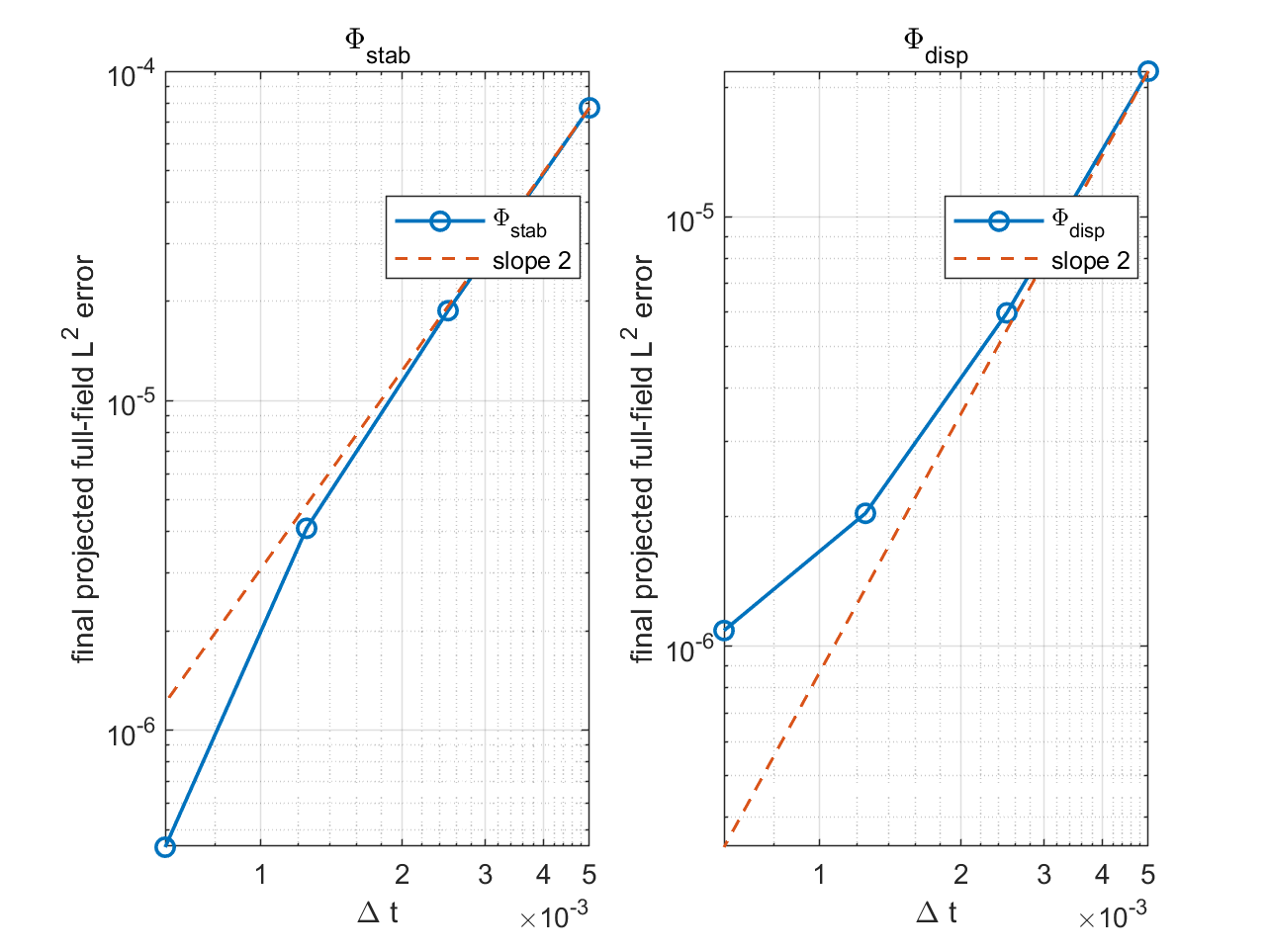}
\caption{Temporal convergence behavior of $\Phi_{\rm stab}$ (left) and $\Phi_{\rm disp}$ (right) on the fixed grid $N=80$. The reference slope confirms second-order convergence for the projected physical field.}
\label{fig:1conver}
\end{figure}

Finally, Fig.~\ref{fig:1disp} compares the numerical phase error. For these tests we use $N=60$ and $s=0.5$. The reduction produced by $\Phi_{\rm disp}$ agrees with Theorem~\ref{thm:disp}; in contrast, $\Phi_{\rm stab}$ is designed to maximize the real-coefficient stability interval.

\begin{figure}[htbp]
\centering
\includegraphics[width=0.80\textwidth]{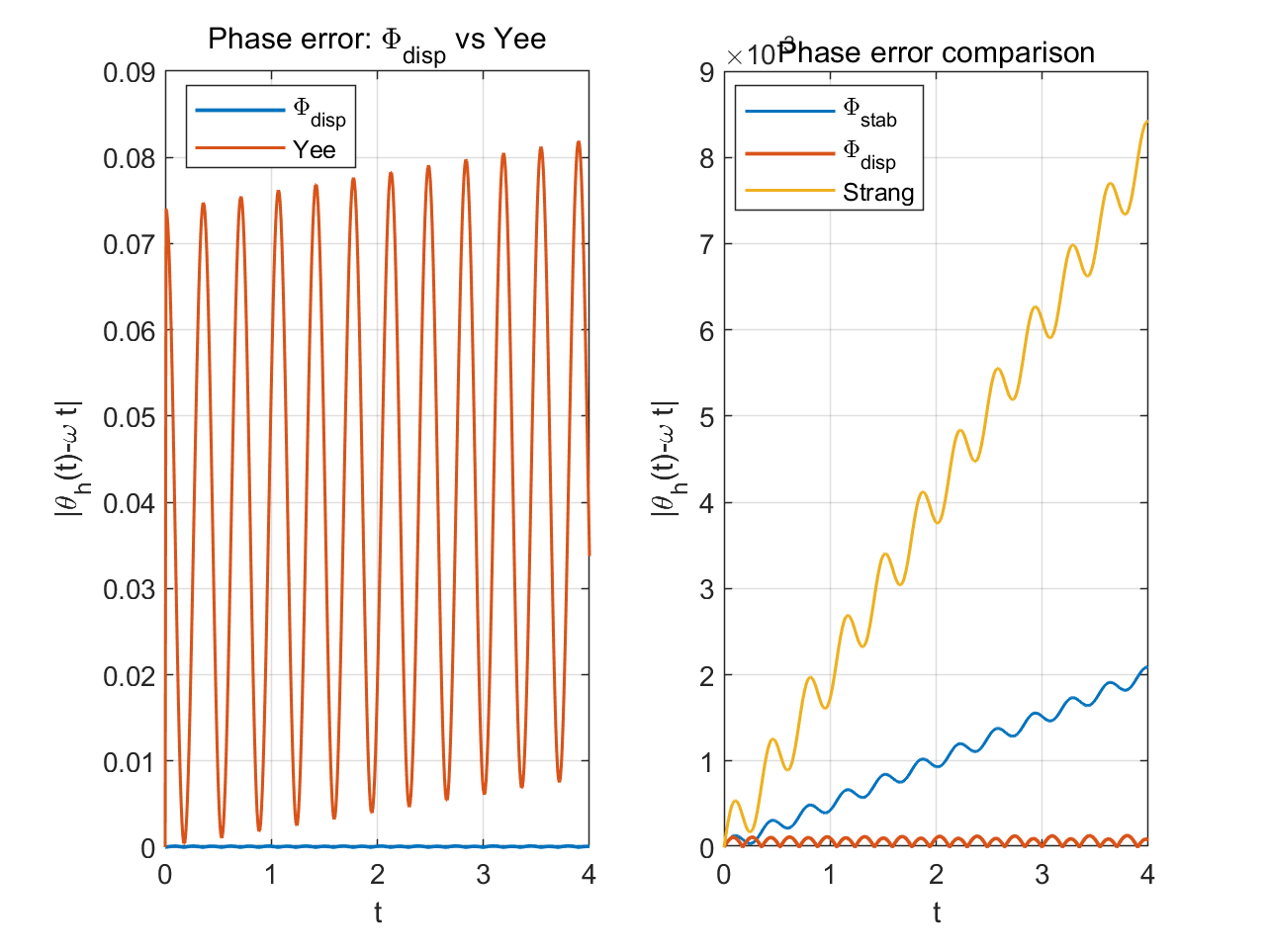}
\caption{Numerical phase error for $N=60$ and CFL number $s=0.5$. Left: comparison of $\Phi_{\rm disp}$ and Yee. Right: comparison of $\Phi_{\rm stab}$, $\Phi_{\rm disp}$, and Strang splitting. The phase-optimized method exhibits the smallest accumulated phase error.}
\label{fig:1disp}
\end{figure}

\section{Conclusions}\label{sec:conclusion}
We have analyzed a low-stage electric--magnetic Maxwell splitting as a one-parameter coefficient-design problem. The common Fourier amplification matrix reveals two distinct optimization objectives. Within the real palindromic family, $a=1/4$ is the unique coefficient maximizing the contiguous spectral CFL interval, giving $s_{\rm stab}^*=12/(7\sqrt d)$ for the fourth-order staggered stencil. This value recovers an earlier optimized symplectic-FDTD coefficient set~\cite{kusaf2005}; the contribution here is its sharp extremal characterization within the specified family.

The same reduction exposes a structural barrier for temporal phase matching: every real coefficient satisfies $g_2\le1/16$, whereas cancellation of the leading temporal phase defect requires $g_2=1/12$. The resulting complex-conjugate coefficients give fourth-order temporal phase accuracy for each fixed semidiscrete Fourier mode. Their CFL threshold is $s_{\rm disp}^*=6\sqrt3/(7\sqrt d)$. The complete complex field update nevertheless remains second order, as follows from the off-diagonal expansion and the semidiscrete convergence result. For real Maxwell data, the physically reported approximation is $\operatorname{Re}U_h^n$; this output projection is independent of the chosen conjugate branch and preserves the second-order error bound.

The analysis is deliberately limited to constant material parameters and uniform staggered grids. The phase-optimized method incurs the cost of an auxiliary imaginary component: it may be implemented with native complex arithmetic or, exactly equivalently, as a doubled real system. The latter representation makes clear that deleting the imaginary coefficient is not a cost-free simplification but a change of integrator that destroys the phase-matching condition. Natural extensions include a uniform space--time error analysis, coefficient design for variable material parameters and nonperiodic boundary closures, and multi-objective optimization that accounts simultaneously for stability, phase accuracy, structural preservation, and arithmetic cost.

\section*{Funding}
The work of Hui Duan was supported by the State Key Laboratory of Scientific and Engineering Computing, Chinese Academy of Sciences. The work of Hongliang Li was supported by the National Natural Science Foundation of China [grant number: 12371438 and 12101438].

\end{document}